\documentclass[a4paper, 12pt, oneside]{article}
\pdfoutput=1 
\usepackage[T1]{fontenc}
\usepackage[british]{babel}
\usepackage{lmodern}
\usepackage{etoolbox}
\usepackage{amsmath, amsthm, amssymb, bm, mleftright}
\usepackage{booktabs, longtable, caption, subcaption, multirow}
\usepackage[space]{grffile}
\usepackage[margin=2.5cm, footskip=1cm]{geometry}
\usepackage{enumitem}
\setenumerate[1]{label=(\arabic*), ref=\arabic*}
\usepackage[protrusion=allmath]{microtype}
\usepackage{tikz}
\usetikzlibrary{cd, calc}
\usepackage{listings}
\usepackage[colorlinks, citecolor=blue, linkcolor=blue, linktoc=section]{hyperref}

\allowdisplaybreaks
\usepackage[capitalise, noabbrev]{cleveref}
\creflabelformat{enumi}{(#2#1#3)}
\creflabelformat{enumii}{(#2#1#3)}
\creflabelformat{enumiii}{(#2#1#3)}

\usepackage{titlesec}
\titleformat*{\section}{\large\bfseries}
\titleformat*{\subsection}{\normalsize\bfseries}
\newlength{\VerticalSpaceAfterParagraph}
\titlespacing*{\paragraph}{0pt}{\VerticalSpaceAfterParagraph}{1em}

\usepackage{titling}
\pretitle{\vspace{-\baselineskip}\begin{center}\Large\bfseries}
\posttitle{\end{center}\vspace{-0.25\baselineskip}}
\preauthor{\begin{center}}
\postauthor{\end{center}}
\predate{\begin{center}}
\postdate{\end{center}}

\setlist
  {
    topsep = 5.0pt plus 2.0pt minus 3.0pt,
    partopsep = 1.5pt plus 1.0pt minus 1.0pt,
    parsep = 2.5pt plus 1.25pt minus 0.5pt,
    itemsep = 0pt plus 1.25pt minus 0.5pt
  }

\theoremstyle{plain}
\newtheorem{theorem}{Theorem}
\newtheorem{proposition}[theorem]{Proposition}
\newtheorem{lemma}[theorem]{Lemma}
\newtheorem{corollary}[theorem]{Corollary}
\newtheorem{question}[theorem]{Question}

\theoremstyle{definition}
\newtheorem{definition}[theorem]{Definition}
\newtheorem{example}[theorem]{Example}
\newtheorem{problem}[theorem]{Problem}
\newtheorem{notation}[theorem]{Notation}
\newtheorem{construction}[theorem]{Construction}

\theoremstyle{remark}
\newtheorem{remark}[theorem]{Remark}

\newcommand\blfootnote[1]
  {%
    \begingroup%
      \renewcommand\thefootnote{}%
      \footnote{#1}%
      \addtocounter{footnote}{-1}%
    \endgroup%
  }

\newcommand\mapsfrom{\mathrel{\reflectbox{\ensuremath{\mapsto}}}}

\crefname{(cond)}{condition}{conditions}
\Crefname{(cond)}{Condition}{Conditions}
\creflabelformat{(cond)}{(#2#1#3)}

\usepackage[titles]{tocloft}
\numberwithin{theorem}{section}
\numberwithin{equation}{theorem}

\setenumerate[2]{label=(\arabic{enumi}\alph*), ref=\arabic{enumi}\alph*}

\crefname{problem}{Problem}{Problems}

\title{Counting divisorial contractions with centre a $cA_n$-singularity}
\author{Erik Paemurru}
\date{4th~July 2024}
\newcommand\keywords{Weighted blowup, terminal singularity, Mori theory}
\newcommand\subjclass{14E30 (Primary) 14E05, 14J30}

\begin{document}

\maketitle

\blfootnote{To appear in Publications of the Research Institute for Mathematical Sciences, Kyoto University. Communicated by Y.~Namikawa. Received November 1, 2022. Revised February 23, 2023.}

\begin{abstract}
First, we simplify the existing classification due to Kawakita and Yamamoto of 3-dimensional divisorial contractions with centre a $cA_n$-singularity, also called compound $A_n$ singularity. Next, we describe the global algebraic divisorial contractions corresponding to a given local analytic equivalence class of divisorial contractions with centre a point. Finally, we consider divisorial contractions of discrepancy at least 2 to a fixed variety with centre a $cA_n$-singularity. We show that if there exists one such divisorial contraction, then there exist uncountably many such divisorial contractions.
\blfootnote{\textup{2020} \textit{Mathematics Subject Classification}. \subjclass{}.}%
\end{abstract}

\blfootnote{Mathematik und Informatik, Universität des Saarlandes, 66123 Saarbrücken, Germany.\\
\textit{Former address:} Department of Mathematics, University of Miami, Coral Gables, Florida 33146, USA.}
\blfootnote{\textit{e-mail:} \texttt{paemurru@math.uni-sb.de}}

\tableofcontents

\section{Introduction}

The minimal model program and the Sarkisov program give a general framework for the birational classification of algebraic varieties, a central problem in algebraic geometry. Morphisms called divisorial contractions play a major role in both the minimal model program and the Sarkisov program. Therefore, classifying divisorial contractions is a fundamental problem.

A \emph{divisorial contraction} is a proper birational morphism $\varphi\colon Y \to X$ between terminal algebraic varieties such that the exceptional locus of $\varphi$ is a prime divisor and $-K_Y$ is $\varphi$-ample. The explicit classification of 3-dimensional divisorial contractions where the centre is a point has been completed, except when the centre is a $cD_n$ or a $cE_n$-singularity and the discrepancy is~1, for which there are only unpublished manuscripts \cite{Haya,Hayb}. The case where the centre is a non-Gorenstein point has been done in \cite{Hay99,Hay00,Hay05,Kaw05,Kaw12,Kaw96} and the Gorenstein case in \cite{Kaw01,Kaw02,Kaw03}, \cite[Theorem~1.2]{Kaw05} and \cite{Yam18}.

Above, the divisorial contractions are classified up to \emph{local analytic equivalence}, meaning that we allow local analytic changes on $X$ around $P$ and on $Y$ around the exceptional locus. Since the local analytic germ of a $\mathbb Q$-factorial variety can be non-factorial, the classification is given more generally for $\mathbb Q$-Gorenstein varieties with terminal singularities without requiring $\mathbb Q$-factoriality. If a morphism $\varphi\colon Y \to X$ is a divisorial contraction in this sense, without requiring $\mathbb Q$-factoriality, and if $X$ is $\mathbb Q$-factorial, then we automatically find that $Y$ is $\mathbb Q$-factorial, since the prime exceptional divisor is Cartier.

In this paper we focus on $cA_n$-singularities, also called compound $A_n$ singularities, meaning that a general section through the point defines the surface $A_n$-singularity, see~\cref{def:pre cAn}. Compound $A_n$ singularities are the simplest 3-dimensional terminal hypersurface singularities. There is an on-going project with the goal of showing that all smooth Fano 3-folds are obtainable by deformations from singular toric 3-folds with $cA_n$-singularities, \cite{CR}.

The classification due to Hayakawa, Kawakita, Kawamata and Yamamoto gives a list of weighted blowups such that every divisorial contraction is locally analytically equivalent to at least one member of the list. One way to improve the classification is to find which members of the classification lists give locally analytically equivalent blowups:

\begin{problem} \label{pro:int loc ana equiv classes}
Describe the local analytic equivalence classes of 3-dimensional divisorial contractions with centre a point.
\end{problem}

This is roughly what was asked in \cite[Problem~3.8]{Cor00}.

We have solved \cref{pro:int loc ana equiv classes} for $cA_n$-singularities in \cref{thm:cou simpler theorem,thm:cou div contr same type are isom}. This can drastically simplify the classification, as the complicated family in \cref{thm:pre div contr cAn}(3) reduces to just one simple case \cref{thm:cou simpler theorem}(3).

The next step is to classify divisorial contractions globally algebraically:

\begin{problem} \label{pro:int global algebraic classification}
Describe all global algebraic blowups up to isomorphism over the base that are locally analytically equivalent to a given weighted blowup.
\end{problem}

We have solved \cref{pro:int global algebraic classification} completely in \cref{thm:a2a constructing weighted blowups}. The global algebraic classification has applications in birational rigidity, finding birational relations and computing Sarkisov links, see \cite{AK16,AZ16,Oka14,Oka18,Oka20,Pae20}.

To prove that a given morphism is a divisorial contraction of a certain type, for example when computing Sarkisov links, it is best to have a classification list where the conditions are as mild and as easy to check as possible. One way to phrase this:

\begin{problem} \label{pro:int algorithm}
Describe an algorithm to determine whether a given weighted blowup is locally analytically equivalent to a given member of the classification list.
\end{problem}

We have solved \cref{pro:int algorithm} for $cA_n$-singularities in \cref{thm:cou simpler theorem,thm:cou div contr same type are isom}. It is straight-forward to determine the weight of a power series and it is algorithmic to check the singularity type of a simple singularity (\cref{def:pre simple singularity}). To check whether a given singularity is of type $A_n$, $D_n$, $E_6$, $E_7$ or $E_8$, see \cite[16.2 The determinator 1.--$9_k$.]{AGZV85} or \cite[Theorems~I.2.48, I.2.51 and I.2.53]{GLS07}. It can be computed using a computer algebra system, for example Singular \cite{Singular}.

On the other hand, to prove local properties or local inequalities such as \cite[Theorem~1.2]{KOPP24}, it helps to have a list which is as specific as possible, containing only few members. Even though the classification lists in the literature contain uncountable families of weighted blowups, a countably infinite list or even a finite list in certain cases may suffice.

\begin{problem} \label{pro:int counts}
Given a variety $X$ and a point~$P$, determine whether there exist finitely many, countably many or uncountably many divisorial contractions to $X$ with centre~$P$, depending on the singularity type of~$P$, where the counting is up to local analytic equivalence and up to global algebraic isomorphism over~$X$.
\end{problem}

We have solved \cref{pro:int counts} for $cA_n$-singularities in \cref{thm:cou counting divisorial contractions}. We have also included the case of smooth points in \cref{tab:int counts}. By the proof of \cref{thm:cou counting divisorial contractions} the cardinalities up to local biholomorphism around the exceptional loci over the base are the same as up to global algebraic isomorphism over the base in the case of $cA_n$-points and smooth points.

\begin{table}[h!]
\centering
\caption{Counting divisorial contractions with centre a $cA_n$-point\label{tab:int counts}}
\begin{tabular}{ccc}
\toprule
singularity & isomorphism over base & local analytic equivalence\\
\midrule
smooth point & uncountable & countable\\
$cA_n$, only discr $1$ & $n$ & $\lceil n/2 \rceil$\\
$cA_n$, admits discr $>1$ & uncountable & finite\\
\bottomrule
\end{tabular}
\end{table}

As an example application of such results, \cite{Oka20} uses the specific counts of divisorial contractions of a given type described in \cite{Hay99} (such as \cite[Theorem~6.4]{Hay99}) to prove birational birigidity of varieties.

\Cref{tab:int counts} raises the following questions:

\begin{question} \label{que:int finitely many up to loc ana}
Let $X$ be a 3-dimensional variety and $P \in X$ a singular point. Do there exist only finitely many divisorial contractions to $X$ with centre $P$ up to local analytic equivalence?
\end{question}

\begin{question} \label{que:int uncountable many non-min disc div contrs}
Let $X$ be a 3-dimensional variety and $P \in X$ a point. Is it true that there exist uncountably many divisorial contractions to $X$ with centre $P$ up to isomorphism over $X$ if and only if there exists a divisorial contraction to $X$ with centre $P$ with discrepancy greater than~$1$?
\end{question}

Regarding \cref{que:int uncountable many non-min disc div contrs}, it is known that there exist only finitely many divisorial contractions of discrepancy at most~$1$ to a fixed variety, see \cref{thm:cou finitely many if disc at most 1}. By \cite[§6]{Pae20} we expect there to be only finitely many divisorial contractions of ordinary type that are $(r_1, r_2, a, 1)$-blowups with centre a $cA_n$-singularity even if the discrepancy $a$ is greater than 1 as long as the inequalities $a \leq r_1 \leq r_2$ are satisfied. This does not answer \cref{que:int uncountable many non-min disc div contrs} negatively, see \cref{thm:cou counting divisorial contractions}.

The proofs in this paper rely on the concept of \emph{weight-respecting maps}, see \cref{def:wei weight-respecting}, which is comparable to the equivalence relation `$\sim$' defined in \cite[3.7 Weighted valuations]{Hay99} for 3-dimensional index $\geq 2$ terminal singularities embedded as hypersurfaces in orbifolds.

\section{Meaning of classification}

The classification due to Hayakawa, Kawakita, Kawamata and Yamamoto is a classification list, as defined in \cref{def:mea classification list}, except that it does not satisfy \cref{itm:int except cD cE} if the discrepancy of $\varphi$ is $1$ and the centre is either a $cD$ or a $cE$ point.

\begin{definition} \label{def:mea classification list}
A set $L$ is called a \textbf{classification list} if it consists of pairs $(\bm w, Z)$, where $\bm w := (w_1, \ldots, w_5) \in ((1/m)\mathbb Z)^5$ is a vector of positive rational numbers and $Z$ is a codimension 2 complete intersection complex analytic space with an isolated singularity at the origin $\bm 0$ inside an orbifold $\mathbb C^5/\mathbb Z_m$, such that
\begin{enumerate}
\item \label{itm:int except cD cE} for every $3$-dimensional divisorial contraction $\varphi$ with centre a point: $\varphi$ is locally analytically equivalent (\cref{def:pre local analytic equivalence}) to the $\bm w$-blowup of $Z$ for some $(\bm w, Z)$ in~$L$,
\item for every pair $(\bm w, Z)$ in $L$: there exists a $\mathbb Q$-Gorenstein variety $X$ with terminal singularities and a point $P \in X$ such that $(X, P)$ is locally biholomorphic to~$(Z, \bm 0)$, and
\item \label{itm:int constructing a divisorial contraction} for every $\mathbb Q$-Gorenstein variety $X$ with terminal singularities and point $P \in X$: if $(X, P)$ is locally biholomorphic to $(Z, \bm 0)$ for some $(\bm w, Z)$ in~$L$, then there exists a divisorial contraction to $X$ which is locally analytically equivalent to the $\bm w$-blowup of~$Z$.
\end{enumerate}
\end{definition}

The divisorial contraction in \cref{itm:int constructing a divisorial contraction} can be constructed using either \cref{thm:a2a constructing projective morphisms} or \cref{thm:a2a constructing weighted blowups}.

Given a classification list $L$ and two pairs $(\bm w, Z)$ and $(\bm w', Z')$, \cref{pro:int loc ana equiv classes} asks to determine when the weighted blowups of $Z$ and $Z'$ are locally analytically equivalent. For $cA_n$-points, we prove their local analytic equivalence if $\bm w = \bm w'$ in \cref{thm:cou div contr same type are isom}. It should not be difficult to prove in the case where there are only finitely many such divisorial contractions, which happens for example when the discrepancy is at most 1, see \cref{thm:cou finitely many if disc at most 1}. See \cite{Hay99,Hay00} for explicit descriptions and counts of minimal discrepancy divisorial contractions with centre a non-Gorenstein point.

\begin{definition}
We say a classification list $L$ is a \textbf{nice classification list} if for every two pairs $(\bm w, Z)$ and $(\bm w', Z')$ in~$L$: the $\bm w$-blowup of $Z$ and the $\bm w'$-blowup of $Z'$ are locally analytically equivalent if and only if $Z$ and $Z'$ are biholomorphic around the origins and $\bm w = \bm w'$.
\end{definition}

Finding a nice classification list, if it exists, would solve \cref{pro:int loc ana equiv classes}.

\begin{question} \label{pro:int decide whether a div contr is in a given class}
Does there exist a nice classification list?
\end{question}

By \cref{thm:cou simpler theorem,thm:cou div contr same type are isom}, the answer to \cref{pro:int decide whether a div contr is in a given class} is `yes' in the case of $cA_n$-points. We give two nice classification lists $L$ and~$L'$ for $cA_n$-singularities and smooth points, corresponding to columns 1, 2, 3 and columns 1, 2, 4, 5 of \cref{tab:int classification}, respectively. The final column `disc' in \cref{tab:int classification} gives the discrepancy of the divisorial contraction.

For $cA_n$-singularities, the classification list contains only weighted blowups such that $X$ is embedded locally analytically as a hypersurface $\mathbb V(f)$ in~$\mathbb C^4$. We can also embed it as a codimension 2 complete intersection $\mathbb V(f, x_5)$ in $\mathbb C^5$ choosing the weight $w_5$ to be any positive integer. In such cases we write only the first four variables $x, y, z, t$ and their weights $w_1, w_2, w_3, w_4$. Similarly $\mathbb C^3$ can be embedded in $\mathbb C^5$ by $\mathbb V(x_4, x_5)$ with any positive integer weights $w_4, w_5$ for $x_4, x_5$, so we give only the first three weights.

The first nice classification list for $cA_n$-singularities and smooth points is given by
\begin{equation} \label{equ:int classification list small}
L := L_1 \cup L_2 \cup L_3 \cup L_4,
\end{equation}
where
\[
\begin{aligned}
L_1 & := \mleft\{ \bigl((1, a, b), \mathbb C^3 \bigr) \;\middle|\; \text{$a$ and $b$ are coprime positive integers, $a \leq b$} \mright\}\\
L_2 & := \mleft\{ \bigl((r_1, r_2, a, 1), \mathbb V(xy + g(z, t)) \bigr) \;\middle|\; \operatorname{wt} g = r_1 + r_2 \mright\}\\
L_3 & := \mleft\{ \bigl((1, 5, 3, 2), \mathbb V(xy + z^2 + t^3) \bigr) \mright\}\\
L_4 & := \mleft\{ \bigl((4, 3, 2, 1), \mathbb V(x^2 + y^2 + z^3 + xt^2) \bigr) \mright\},
\end{aligned}
\]
where in $L_2$ the convergent power series $g \in \mathbb C\{z, t\}$ defines an isolated singularity at the origin and $r_1, r_2$ and $a$ are positive integers that satisfy $r_1 \leq r_2$, $a$ divides $r_1 + r_2$, $a$ is coprime to both $r_1$ and $r_2$, and $a(n + 1) = r_1 + r_2$. The polynomial $xy + g(z, t)$ with $\operatorname{wt} g = r_1 + r_2$ appears in \cite[Theorem~1.1]{Kaw02} and the polynomial $xy + z^2 + t^3$ appears in \cite[Theorem~1.13]{Kaw03}, whereas the complicated condition of \cite[Theorem~2.6]{Yam18} is simplified in \cref{thm:cou simpler theorem} to $x^2 + y^2 + z^3 + xt^2$.

The second nice classification list for $cA_n$-singularities and smooth points is given by
\[
L' := L_1 \cup L_2' \cup L_3' \cup L_4',
\]
where
\begin{align*}
L_2' & := \mleft\{ \bigl((r_1, r_2, a, 1), \mathbb V(f) \bigr) \;\middle|\; \text{$(\mathbb V(f), \bm 0)$ is a $cA_n$-singularity},\, \operatorname{wt} f = r_1 + r_2 \mright\}\\
L_3' & := \mleft\{ \bigl((1, 5, 3, 2), \mathbb V(f) \bigr) \;\middle|\; \text{$(\mathbb V(f), \bm 0)$ is an $A_2$-singularity},\, \operatorname{wt} f = 6 \mright\}\\
L_4' & := \mleft\{ \bigl((4, 3, 2, 1), \mathbb V(f) \bigr) \;\middle|\; \text{$(\mathbb V(f), \bm 0)$ is an $E_6$-singularity},\, \operatorname{wt} f = 6 \mright\},
\end{align*}
where the convergent power series $f \in \mathbb C\{x, y, z, t\}$ defines an isolated singularity at the origin and in $L_2$ the positive integers $r_1, r_2$ and $a$ satisfy the same conditions as for the first classification list. The singularities $cA_n$, $A_2$ and $E_6$ are defined in \cref{def:pre simple singularity,def:pre cAn}.

\newcommand\mtext[1]{\begin{minipage}{.22\textwidth} \vspace{4pt plus 2pt} #1 \vspace{4pt plus 2pt} \end{minipage}}
\newcommand\stext[1]{\begin{minipage}{.16\textwidth} \vspace{4pt plus 2pt} #1 \vspace{4pt plus 2pt} \end{minipage}}

\newcommand\verticalspace[1]{\begin{minipage}{.01\textwidth}\vspace{#1\baselineskip} \mbox{ }\end{minipage}}

\setlength\tabcolsep{0.015\textwidth}
\begin{table}[h!]
\centering
\caption{Local analytic equivalence classes of divisorial contractions, $cA_n$-singularities\label{tab:int classification}}
\begin{tabular}{cccccc}
\toprule
$\bm w$ & conditions & $f$ & $\operatorname{wt} f$ & sing & discr\\
\midrule
$(1, a, b)$ & $(a, b) = 1,~ a \leq b$ &  & \verticalspace{1.5} & sm & $a+b$\\
$(r_1, r_2, a, 1)$ & \mtext{$r_1 \leq r_2$, $a \mid r_1 + r_2$, $(r_1, a) = (r_2, a) = 1$, $a(n + 1) = r_1 + r_2$} & \stext{$xy + g(z, t)$,\\$\operatorname{wt} g = r_1 + r_2$} & $r_1 + r_2$ & $cA_n$ & $a$\\
$(1, 5, 3, 2)$ && $xy + z^2 + t^3$ & $6$ & $A_2$ & $4$\\
$(4, 3, 2, 1)$ && $x^2 + y^2 + z^3 + xt^2$ & $6$ & $E_6$ & $3$\\
\bottomrule
\end{tabular}
\end{table}

\renewcommand{\arraystretch}{1}

\section{Preliminaries}

\begin{notation}
Let $\mathbb C$ denote the complex numbers. A \textbf{variety}, short for algebraic variety, is defined to be an integral separated scheme of finite type over~$\mathbb C$. All morphisms of varieties are defined over~$\mathbb C$. The $\mathbb C$-algebra of power series that are absolutely convergent in a neighbourhood of the origin is denoted $\mathbb C\{\bm x \}$, short for $\mathbb C\{x_1, \ldots, x_n\}$. The complex space, short for complex analytic space, corresponding to a variety $X$ is denoted $X^{\mathrm{an}}$. A~\textbf{singularity} is defined to be a complex space germ (see \cite[Definition~I.1.47]{GLS07}).
If $I$ is an ideal of regular functions on a variety or an ideal of holomorphic functions on a complex space, then $\mathbb V(I)$ denotes the zero locus of~$I$.
If $I$ is an ideal of holomorphic function germs on a complex space germ $(X, P)$, then $(\mathbb V(I), P)$ denotes the (possibly non-reduced) subgerm defined by~$I$ (see \cite[§I.1.4]{GLS07}).

Given a convergent power series $f \in \mathbb C\{\bm x\} := \mathbb C\{x_1, \ldots, x_n\}$ we define the \textbf{multiplicity} of $f$, denoted $\operatorname{mult} f$, by
\[
\operatorname{mult} f := \min \{ i_1 + \ldots + i_n \mid \text{$x_1^{i_1} \cdot \ldots \cdot x_n^{i_n}$ has non-zero coefficient in $f$} \}.
\]
Given positive integer weights $w_1, \ldots, w_n$ for variables $x_1, \ldots, x_n$ we define the \textbf{weight} of~$f$, denoted $\operatorname{wt} f$, by
\[
\operatorname{wt} f := \min \{ w_1 i_1 + \ldots + w_n i_n \mid \text{$x_1^{i_1} \cdot \ldots \cdot x_n^{i_n}$ has non-zero coefficient in $f$} \}
\]
if $f$ is non-zero, and we define $\operatorname{wt} 0 = \infty$ otherwise. We denote the quasihomogeneous weight $d$ part of $f$ by $f_{\operatorname{wt} = d}$. The \textbf{quadratic part} of $f$ is defined to be the homogeneous degree 2 part of~$f$. The \textbf{quadratic rank} of $f$ is defined to be the rank of the symmetric matrix $M$ with complex coefficients such that the quadratic part of $f$ is equal to $\bm x^T M \bm x$ where $\bm x$ is the vector $(x_1, \ldots, x_n)$.

We denote the \textbf{Jacobian ideal $\mleft( \partial f/\partial x_1, \ldots, \partial f/\partial x_n \mright) \subseteq \mathbb C\{\bm x\}$} of $f$ by~$j(f)$. We remind that the \textbf{Milnor algebra} of $f$ is the $\mathbb C$-algebra $\mathbb C\{\bm x\} / j(f)$. We say that a set $S$ of monomials of $\mathbb C[\bm x] := \mathbb C\{x_1, \ldots, x_n\}$ is a \textbf{monomial spanning set} for a $\mathbb C$-algebra $\mathbb C\{\bm x\} / J$ if the set $\{s + J \mid s \in S\}$ generates the $\mathbb C$-vector space $\mathbb C\{\bm x\} / J$, and we say $S$ is a \textbf{monomial basis} for the $\mathbb C$-algebra $\mathbb C\{\bm x\}/J$ if $\{s + J \mid s \in S\}$ is a basis for the $\mathbb C$-vector space $\mathbb C\{\bm x\}/J$. By Zorn's Lemma every $\mathbb C$-algebra $\mathbb C\{\bm x\}/J$ has a (possibly infinite) monomial basis.
\end{notation}

\begin{definition} \label{def:pre divisorial contraction}
A \textbf{divisorial contraction} is a proper birational morphism $\varphi\colon Y \to X$ between normal $\mathbb Q$-Gorenstein varieties with terminal singularities such that
\begin{enumerate}
\item the exceptional locus of $\varphi$ is a prime divisor and
\item $-K_Y$ is $\varphi$-ample.
\end{enumerate}
\end{definition}

\begin{definition}[{\cite[Definition~2.14]{Pae21}}] \label{def:pre local analytic equivalence}
Let $\varphi\colon Y \to X$ and $\varphi'\colon Y' \to X'$ be birational morphisms of varieties (or bimeromorphic holomorphisms of complex analytic spaces). We say that an isomorphism $X \to X'$ \textbf{lifts} if there exists an isomorphism $Y \cong Y'$ such that the diagram
\begin{equation*}
\begin{tikzcd}
Y \arrow[r, ""]{r} \arrow[d, "\varphi"]{r} & Y' \arrow[d, "\varphi'"]{r}\\
X \arrow[r, ""]{r} & X'
\end{tikzcd}
\end{equation*}
commutes. We say that $\varphi$ and $\varphi'$ are \textbf{equivalent} if there exist an isomorphism $X \cong X'$ that lifts. We say $\varphi$ and $\varphi'$ are \textbf{locally equivalent} if there exist isomorphic open subsets $U \subseteq X$ and $U' \subseteq X'$ containing the centres of the morphisms $\varphi$ and $\varphi'$ such that the restrictions $\mleft.\varphi\mright|_{\varphi^{-1}U}\colon \varphi^{-1}U \to U$ and $\mleft.\varphi'\mright|_{\varphi'^{-1}U'}\colon \varphi'^{-1}U' \to U'$ are equivalent.
\end{definition}

If we consider the complex space corresponding to a variety or when we wish to emphasize that we are working in the category of complex spaces, then we say \textbf{analytically equivalent} or \textbf{locally analytically equivalent}.

\begin{definition} \label{def:pre wei blup}
Let $n$ be a positive integer and let $\bm w = (w_1, \ldots, w_n)$ be positive integers, called the weights of the blowup. Define a $\mathbb C^*$-action on $\mathbb C^{n+1}$ by $\lambda \cdot (u, x_1, \ldots, x_n) = (\lambda^{-1} u, \lambda^{w_1} x_1, \ldots, \lambda^{w_n} x_n)$ and define $T$ by the geometric quotient $(\mathbb C^{n+1} \setminus \mathbb V(x_1, \ldots, x_n)) / \mathbb C^*$ (or its analytification). Then the map $\varphi\colon T \to \mathbb C^n$, $[u, x_1, \ldots, x_n] \mapsto (u^{w_1} x_1, \ldots, u^{w_n} x_n)$ is called the \emph{$\bm w$-blowup of~$\mathbb C^n$}. If $Z \subseteq \mathbb C^n$ is a closed subvariety (or a closed complex subspace $Z \subseteq D$ where $D \subseteq \mathbb C^n$ is open) and $\tilde Z$ is the closure of $\varphi^{-1}(Z \setminus \{\bm 0\})$ in $T$ (in~$\varphi^{-1} D$), then the restriction $\mleft.\varphi\mright|_{\tilde Z}\colon \tilde Z \to Z$ is called the \emph{$\bm w$-blowup of~$Z$}. Let $\rho\colon Y \to X$ be a surjective birational morphism of varieties (or a surjective bimeromorphic holomorphism of complex spaces). Given an open subset $U \subseteq X$ containing the centre of $\rho$ and an isomorphism $U \cong X' \subseteq \mathbb C^n$ taking a point $P \in X$ to the origin~$\bm 0$, the map $\rho$ is called the \textbf{$\bm w$-blowup of $X$ at~$P$} if the restriction $\mleft.\rho\mright|_{\rho^{-1}U}\colon \rho^{-1}U \to U$ is equivalent, through the given isomorphism $U \cong X'$, to the $\bm w$-blowup of~$X'$.
\end{definition}

\begin{remark}
\begin{enumerate}[label=(\alph*), ref=\alph*]
\item A weighted blowup crucially depends both on the isomorphism $U \cong X'$ and a choice of coordinates $x_1, \ldots, x_n$, even though it is not explicit in the notation.
\item Replacing $\bm w$ by $(w_1/g, \ldots, w_n/g)$ in \cref{def:pre wei blup}, where $g$ is the greatest common divisor of $w_1, \ldots, w_n$, gives an isomorphic blowup over~$X$.
\item By \cite[Theorem~5.1.11]{CLS11}, the weighted blowup of an affine space in \cref{def:pre wei blup} coincides with the toric description of subdividing a cone in \cite[Proposition-Definition~10.3]{KM92}.
\end{enumerate}
\end{remark}

\begin{definition} \label{def:pre simple singularity}
A \textbf{simple hypersurface singularity}, also known as an \emph{ADE-singularity}, is a complex space germ $(X, P)$ isomorphic to $(\mathbb V(f), \bm 0)$ where $f \in \mathbb C\{x_1, \ldots, x_n\}$ is one of the following:
\begin{itemize}
\item[$A_k$:] $x_1^{k+1} + x_2^2 + \ldots + x_n^2$, \quad $k \geq 1$,
\item[$D_k$:] $x_1(x_2^2 + x_1^{k-2}) + x_3^2 + \ldots + x_n^2$, \quad $k \geq 4$,
\item[$E_6$:] $x_1^3 + x_2^4 + x_3^2 + \ldots + x_n^2$,
\item[$E_7$:] $x_1(x_1^2 + x_2^3) + x_3^2 + \ldots + x_n^2$,
\item[$E_8$:] $x_1^3 + x_2^4 + x_3^2 + \ldots + x_n^2$,
\end{itemize}
where $n$ is at least 1 in case $A_k$ and at least 2 in all other cases.
\end{definition}

\begin{definition} \label{def:pre cAn}
Let $k$ be a positive integer. A \textbf{compound $A_k$-singularity}, denoted $cA_k$, is a complex space germ isomorphic to $(\mathbb V(xy + g), \bm 0) \subseteq (\mathbb C^4, \bm 0)$, where $g \in \mathbb C\{z, t\}$ has multiplicity $k+1$.
\end{definition}

We state the known classification of divisorial contractions to both smooth points and $cA_n$-points.

\begin{theorem}[{\cite[Theorem~1.1]{Kaw01}}] \label{thm:pre div contr smooth}
Let $P$ be a smooth point of a $3$-dimensional $\mathbb Q$-Gorenstein variety $X$ with terminal singularities. Let $\varphi\colon Y \to X$ be a surjective birational morphism with centre~$P$. Then $\varphi$ is a divisorial contraction if and only if $\varphi$ is locally analytically equivalent to the $(1, a, b)$-blowup of~$\mathbb A^3$, where $a$ and $b$ are coprime positive integers.
\end{theorem}

\begin{remark}
By \cite[Lemma~11.4.10]{CLS11} the discrepancy of the $(1, a, b)$-blowup of $\mathbb A^3$ is $a + b$.
\end{remark}

\begin{theorem} \label{thm:pre div contr cAn}
Let $n$ be a positive integer. Let $P$ be a $cA_n$-point of a $\mathbb Q$-Gorenstein variety with terminal singularities. Let $\varphi\colon Y \to X$ be a surjective birational morphism with centre~$P$. Then $\varphi$ is a divisorial contraction if and only if one of the following holds:
\begin{enumerate}
\item \label[(cond)]{itm:pre div contr cAn - gen} $\varphi$ is locally analytically equivalent to the $(r_1, r_2, a, 1)$-blowup of $\mathbb V(f)$ at $\bm 0$ where $f \in \mathbb C\{x, y, z, t\}$ is such that
  \begin{enumerate}
  \item $r_1, r_2$ and $a$ are positive integers such that $r_1 \leq r_2$, $a(n+1) = r_1 + r_2$, $a$ divides $r_1 + r_2$, $a$ is coprime to both $r_1$ and $r_2$ and
  \item \label[(cond)]{itm:pre div contr cAn - gen f} $f = xy + g(z, t)$ where $\operatorname{wt} g = r_1 + r_2$,
  \end{enumerate}
\item \label[(cond)]{itm:pre div contr cAn - cA1} $n = 1$ and $\varphi$ is locally analytically equivalent to the $(1, 5, 3, 2)$-blowup of $\mathbb V(f)$ at $\bm 0$ where $f \in \mathbb C\{x, y, z, t\}$ is such that
  \begin{enumerate}
  \item \label[(cond)]{itm:pre div contr cAn - cA1 f} $f = xy + z^2 + t^3$
  \end{enumerate}
\item \label[(cond)]{itm:pre div contr cAn - cA2} $n = 2$ and $\varphi$ is locally analytically equivalent to the $(4, 3, 2, 1)$-blowup of $\mathbb V(f)$ at $\bm 0$ where $f \in \mathbb C\{x, y, z, t\}$ is such that
  \begin{enumerate}
  \item \label[(cond)]{itm:pre div contr cAn - cA2 f} $f = x^2 + y^2 + 2cxy + 2xp(z, t) + 2cyp_{\operatorname{wt}=3}(z, t) + z^3 + g(z, t)$ where $c \in \mathbb C \setminus \{-1, 1\}$, $\operatorname{wt} g \geq 6$, the power series $p$ contains only monomials of weight $2$ and $3$ for the weights $(4, 3, 2, 1)$, the coefficient of $t^2$ is non-zero in $p$ and $\deg g(z, 1) \leq 2$.
  \end{enumerate}
\end{enumerate}
\end{theorem}

\Cref{itm:pre div contr cAn - gen f,itm:pre div contr cAn - cA1 f,itm:pre div contr cAn - cA2 f} are the same as in \cite[Theorem~1.1]{Kaw02}, \cite[Theorem~1.13]{Kaw03} and \cite[Theorem~2.6]{Yam18} except for the small difference in that the condition that \emph{$z^{(r_1 + r_2)/a}$ has a non-zero coefficient in~$f$} is replaced by the equivalent condition $a(n+1) = r_1 + r_2$. We simplify \cref{thm:pre div contr cAn} in \cref{thm:cou simpler theorem}. In particular, we show that we can replace \cref{itm:pre div contr cAn - cA2 f} with the much simpler condition $f = x^2 + y^2 + z^3 + x t^2$. This polynomial appears in \cite[Example~6.8]{Kaw03}.

\begin{remark}
By \cite[§3.9]{Hay99} the discrepancy in \cref{thm:pre div contr cAn} in cases \labelcref{itm:pre div contr cAn - gen}, \labelcref{itm:pre div contr cAn - cA1} and \labelcref{itm:pre div contr cAn - cA2} is respectively $a$, $4$ and $3$.
\end{remark}

\section{Weight-respecting maps}

The main tools we use in this paper are weight-respecting maps (see \cref{def:wei weight-respecting}) and some classical theorems from singularity theory in weight-respecting form (see \cref{thm:wei wei-resp power series to poly,thm:wei x1*x2 weight-respecting splitting,thm:wei wei-resp reduce simple sing to normal form}).

For \cref{def:wei weight-respecting,thm:wei wei-resp lifts}, let $n$ and $m$ be positive integers. Let $\bm x = (x_1, \ldots, x_n)$ and $\bm y = (y_1, \ldots, y_m)$ denote the coordinates on $\mathbb C^n$ and $\mathbb C^m$, respectively. Choose positive integer weights for $\bm x$ and~$\bm y$.

\begin{definition}[{\cite[Definition~4.1]{Pae21}}] \label{def:wei weight-respecting}
Let $X \subseteq \mathbb C^n$ and $X' \subseteq \mathbb C^m$ be complex analytic spaces. We say that a biholomorphic map $\psi\colon X \to X'$ taking $\bm 0$ to $\bm 0$ is \textbf{weight-respecting} if denoting its inverse by~$\theta$, we can locally analytically around the origins write $\psi = (\psi_1, \ldots, \psi_m)$ and $\theta = (\theta_1, \ldots, \theta_n)$ where for all $i$ and~$j$, the power series $\psi_j \in \mathbb C\{\bm x\}$ and $\theta_i \in \mathbb C\{\bm y\}$ satisfy $\operatorname{wt}(\psi_j) \geq \operatorname{wt}(y_j)$ and $\operatorname{wt}(\theta_i) \geq \operatorname{wt}(x_i)$.
\end{definition}

Compare \cref{def:wei weight-respecting} with the equivalence relation `$\sim$' defined in \cite[3.7 Weighted valuations]{Hay99} for 3-dimensional index $\geq 2$ terminal singularities embedded as hypersurfaces in orbifolds.

\begin{lemma}[{\cite[Corollary~4.4]{Pae21}}] \label{thm:wei wei-resp lifts}
If a biholomorphism from $X \subseteq \mathbb C^n$ to $X' \subseteq \mathbb C^m$ taking $\bm 0$ to~$\bm 0$ is weight-respecting, then it lifts to the weighted blown-up spaces.
\end{lemma}

Compare \cref{thm:wei wei-resp lifts} with \cite[Lemma~5.6]{Hay99}.

We remind that two convergent power series $f, g \in \mathbb C\{x_1, \ldots, x_n\}$ are said to be \textbf{right equivalent} if there exists a biholomorphic map germ $\varphi\colon (\mathbb C^n, \bm 0) \to (\mathbb C^n, \bm 0)$ such that $g = f \circ \varphi$, and are said to be \textbf{contact equivalent} if there exists a biholomorphic map germ $\varphi\colon (\mathbb C^n, \bm 0) \to (\mathbb C^n, \bm 0)$ and a unit $u \in \mathbb C\{x_1, \ldots, x_n\}$ such that $g = u(f \circ \varphi)$ (this is \cite[Definition~I.2.9]{GLS07}).

It is a standard result that every convergent power series is right equivalent to a polynomial (see the theorem in \cite[§6.3]{AGZV85} or \cite[Corollary~I.2.24]{GLS07}). By following the standard proof we show that, unsurprisingly, there also exists a weight-respecting right equivalence.

\begin{lemma} \label{thm:wei wei-resp power series to poly}
Let $n$ be a positive integer. Let $f \in \mathbb C\{x_1, \ldots, x_n\}$ define an isolated singularity at the origin with Milnor number~$\mu$. Then for every integer $N \geq \mu + 1$ there exists an automorphism $\Psi$ of $\mathbb C\{x_1, \ldots, x_n\}$ such that
\begin{enumerate}
\item $\Psi(f)$ is equal to the truncation of $f$ up to degree $N$ and
\item for every $i \in \{1, \ldots, n\}$ the truncation of $\Psi(x_i)$ up to degree $N - \mu$ is equal to $x_i$.
\end{enumerate}
\end{lemma}

\begin{proof}
Denote $\mathbb C\{\bm x\} := \mathbb C\{x_1, \ldots, x_n\}$ and let $\mathfrak m := \sum x_i \mathbb C\{\bm x\}$ be the maximal ideal. The ideal $\mathfrak m^\mu$ is contained in the Jacobian ideal
\[
j(f) := \sum \frac{\partial f}{\partial x_i} \mathbb C\{\bm x\}
\]
of $f$ by the proof of the lemma in \cite[§5.5]{AGZV85}. Let $h \in \mathfrak m^{N+1}$. Define $F \in \mathbb C\{\bm x\}[t]$ by $F := f + th$. Below we construct a biholomorphic map germ $\psi\colon (\mathbb C^n, \bm 0) \to (\mathbb C^n, \bm 0)$ by following the proof of \cite[Theorem~I.2.23]{GLS07}.

First, we show that for every $t_0 \in \mathbb C$ we have the equality of ideals
\begin{equation} \label{equ:wei ideals main}
j(f) \cdot \mathbb C\{\bm x, t - t_0\} = \sum \frac{\partial F}{\partial x_i} \mathbb C\{\bm x, t - t_0\},
\end{equation}
where we consider $F = f + t_0 h + (t - t_0)h$ as an element of $\mathbb C\{\bm x, t - t_0\}$. Since $\mathfrak m^\mu$ is inside $j(f)$, we find the equality of ideals
\begin{equation} \label{equ:wei ideals helper}
j(f) \cdot \mathbb C\{\bm x, t - t_0\} = \sum \frac{\partial F}{\partial x_i} \mathbb C\{\bm x, t - t_0\} + \mathfrak m \cdot j(f) \cdot \mathbb C\{\bm x, t - t_0\}.
\end{equation}
Applying Nakayama Lemma (\cite[Proposition~B.3.6]{GLS07}) to \cref{equ:wei ideals helper} gives \cref{equ:wei ideals main}.

Now, \cref{equ:wei ideals main} implies that
\[
h \in \mleft(\mathfrak m^{N+1-\mu} \cdot \mathbb C\{\bm x, t - t_0\}\mright) \cdot \mleft(\sum \frac{\partial F}{\partial x_i} \mathbb C\{\bm x, t - t_0\}\mright).
\]
By \cite[Theorem~I.2.22(2) and Remark~I.2.22.1]{GLS07} for every $t_0 \in \mathbb C$ there exists an open neighbourhood $U_{t_0} \subseteq \mathbb C$ of $t_0$ such that for every $t' \in U_{t_0}$ there exists a biholomorphic map germ $\psi_{t'} \colon (\mathbb C^n, \bm 0) \to (\mathbb C^n, \bm 0)$ such that $x_i \circ \psi_{t'} - x_i \in \mathfrak m^{N+1-\mu}$ and $(f + t'h) \circ \psi_{t'} = f + t_0 h$. Since the interval $[0, 1]$ is compact, there exist finitely many biholomorphic map germs $\psi_{t_1}, \ldots, \psi_{t_k}$ such that the composition $\psi := \psi_{t_k} \circ \ldots \circ \psi_{t_1}$ satisfies $f \circ \psi = f + h$ and $x_i \circ \psi - x_i \in \mathfrak m^{N+1-\mu}$.

Finally, choosing $h$ to be the negative of the sum of the degree $>N$ parts of $f$ and choosing $\Psi$ to be the precomposition by $\psi$ proves the lemma.
\end{proof}

The lemma in \cite[§12.6]{AGZV85} is useful for computing normal forms of singularities. Here we give a weight-respecting version.

\begin{lemma} \label{thm:wei wei-resp normal forms}
Let $n$ be a positive integer and let $\bm w := (w_1, \ldots, w_n)$ be positive integer weights for the variables $x_1, \ldots, x_n$. Let $f \in \mathbb C\{\bm x\}$ define an isolated singularity at the origin. Let $f_0$ denote the least weight non-zero quasihomogeneous part of~$f$. Choose a monomial spanning set $S \subseteq \mathbb C[\bm x]$ for the Milnor algebra $\mathbb C\{\bm x\}/j(f_0)$ of~$f_0$.
Then there exists an automorphism $\Psi$ of $\mathbb C\{\bm x\}$ of the form $\forall i\colon \Psi(x_i) = x_i + g_i$, where each $g_i \in \mathbb C\{\bm x\}$ is either zero or satisfies
$\operatorname{wt} g_i > \operatorname{wt} x_i$, such that every monomial of $\mathbb C[\bm x]$ with weight greater than $\operatorname{wt} f_0$ that does not belong to $S$ has coefficient zero in~$\Psi(f)$.
\end{lemma}

\begin{proof}
We find an automorphism $\Psi'$ of $\mathbb C[[\bm x]]$ satisfying the conditions of the lemma following the proof in \cite[§12.6]{AGZV85}. Next we define an automorphism $\hat\Psi$ of $\mathbb C\{\bm x\}$ by letting $\hat\Psi(x_i)$ be the truncation of $\Psi'(x_i)$ up to some high enough degree~$N$. The automorphism $\hat\Psi$ satisfies the conditions of the lemma except that there might be monomials of weight greater than $N$ that have a non-zero coefficient in $\hat\Psi(f)$. Now using \cref{thm:wei wei-resp power series to poly} we find a suitable~$\Psi$.
\end{proof}

\begin{lemma} \label{thm:wei x1*x2 weight-respecting splitting}
Let $n \geq 2$ be an integer and let $\bm w = (w_1, \ldots, w_n)$ be positive integer weights for variables $\bm x = (x_1, \ldots, x_n)$. Let $f \in \mathbb C\{\bm x\}$ be such that the coefficient of $x_1 x_2$ is non-zero in $f$ and $\operatorname{wt} x_1 x_2 = \operatorname{wt} f$. Then there exists a weight-respecting automorphism $\Psi$ of $\mathbb C\{\bm x\}$ such that the only monomial that belongs to the ideal $(x_1, x_2)$ and has non-zero coefficient in $\Psi(f)$ is $x_1 x_2$.
\end{lemma}

\begin{proof}
See the proof of \cite[Proposition~4.6]{Pae21}.
\end{proof}

\begin{remark}
In the case where $f$ defines an isolated singularity at the origin, \cref{thm:wei x1*x2 weight-respecting splitting} follows also from \cref{thm:wei wei-resp normal forms}.
\end{remark}

One easy corollary of \cref{thm:wei wei-resp normal forms} is the following:

\begin{corollary} \label{thm:wei wei-resp reduce simple sing to normal form}
Let the variables $\bm x = (x_1, \ldots, x_n)$ have positive integer weights $\bm w = (w_1, \ldots, w_n)$. Let the least weight non-zero quasihomogeneous part $f_0$ of $f \in \mathbb C\{\bm x\}$ be one of the five forms described in \cref{def:pre simple singularity}. Then there is a weight-respecting automorphism $\Psi$ of $\mathbb C\{\bm x\}$ such that $\Psi(f) = f_0$.
\end{corollary}

\begin{proof}
Use \cref{thm:wei wei-resp normal forms} with a set $S$ that does not contain any elements of weight greater than~$\operatorname{wt} f$.
\end{proof}

\section{From analytic to algebraic category}

In \cref{thm:a2a constructing projective morphisms}, we show how to extend blowups along points with possibly non-reduced structure (equivalently, blowups along coherent ideal sheaves with cosupport a point) from the analytic category to the algebraic. \Cref{thm:a2a constructing projective morphisms} was explained to me by Masayuki Kawakita.

\begin{proposition} \label{thm:a2a constructing projective morphisms}
Let $X$ be a variety and $\mathcal J$ a coherent $\mathcal O_{X^{\mathrm{an}}}$-ideal sheaf with cosupport a point, where $X^{\mathrm{an}}$ is the analytification of~$X$. Then there exists a coherent $\mathcal O_X$-ideal sheaf $\mathcal I$ such that its analytification is~$\mathcal J$.
\end{proposition}

\begin{proof}
Since the cosupport of $\mathcal J$ is a point~$P$, there exists a positive integer $k$ such that the $k$-th power of the maximal ideal of $\mathcal O_{X^{\mathrm{an}}, P}$ is in the stalk~$\mathcal J_{P}$. The \lcnamecref{thm:a2a constructing projective morphisms} follows.
\end{proof}

We give an alternative construction in \cref{thm:a2a constructing weighted blowups}\labelcref{itm:a2a constructing weighted blowups - exists one} which describes the divisorial contraction as a weighted blowup. \Cref{thm:a2a constructing weighted blowups}\labelcref{itm:a2a constructing weighted blowups - exists one} is a modification of \cite[Lemma~4.9]{Pae21} which was used for explicitly constructing weighted blowups of affine hypersurfaces with a $cA_n$-point.

\begin{construction} \label{con:a2a construction}
Let $U$ be an affine variety $\operatorname{Spec} \mleft( \mathbb C[x_1, \ldots, x_n] / I \mright)$ containing the point $\mathbb V(x_1, \ldots, x_n)$, for some ideal $I \subseteq \mathbb C[x_1, \ldots, x_n]$. Assign positive integer weights $w_1, \ldots, w_m$ to the variables $y_1, \ldots, y_m$ of~$\mathbb C^m$ and assign weights $1, \ldots, 1$ to the variables $x_1, \ldots, x_n$.

Let $\psi\colon (U^{\mathrm{an}}, \bm 0) \to (Z, \bm 0)$ be a local biholomorphism to a complex space $Z \subseteq \mathbb C^m$ containing the origin. Define the variety $\hat U$ by
\[
\hat U\colon \mathbb V(I + (\psi_1^{< w_1} - y_1, \ldots, \psi_m^{< w_m} - y_m)) \subseteq \mathbb A^{n+m} := \operatorname{Spec} \mathbb C[x_1, \ldots, x_n, y_1, \ldots, y_m],
\]
where $\psi_j^{< w_j} \in \mathbb C[x_1, \ldots, x_n]$ denotes the truncation of the $j$-th coordinate power series of $\psi$ up to order~$w_j - 1$. Note that $\hat U$ is isomorphic to~$U$.
\end{construction}

\begin{proposition} \label{thm:a2a constructing wt-resp maps}
In \cref{con:a2a construction}, the local biholomorphism $(\hat U^{\mathrm{an}}, \bm 0) \to (Z, \bm 0)$ given by the composition of $y_j \mapsto y_j + \psi_j^{< w_j} - \psi_j$ and the projection to $\mathbb C^m$ is weight-respecting.
\end{proposition}

\begin{proof}
The local biholomorphism $y_j \mapsto y_j + \psi_j^{< w_j} - \psi_j$ with inverse $y_j \mapsto y_j - \psi_j^{< w_j} + \psi_j$ is clearly weight-respecting. The projection to $\mathbb C^m$ is given by
\[
(x_1, \ldots, x_n, y_1, \ldots, y_m) \mapsto (y_1, \ldots, y_m)
\]
with inverse
\[
(\theta_1, \ldots, \theta_n, y_1, \ldots, y_m) \mapsfrom (y_1, \ldots, y_m),
\]
where $\theta_i \in \mathbb C\{y_1, \ldots, y_m\}$ are convergent power series with constant term zero. We see that for all~$i$, either $\operatorname{wt} \theta_i \geq \operatorname{wt} x_i = 1$ or $\theta_i = 0$. This shows that the projection to $\mathbb C^m$ is weight-respecting.
\end{proof}

\begin{remark}
If any of the weights $w_j$ was zero in \cref{con:a2a construction}, then the truncation $\psi_j^{< w_j}$ might not be a polynomial.
\end{remark}

\begin{lemma} \label{thm:a2a isom over X algebraically iff analytically}
Let $Y_1 \to X$ and $Y_2 \to X$ be birational morphisms of varieties. Then $Y_1$ and $Y_2$ are isomorphic over $X$ if and only if the analytifications $Y_1^\mathrm{an}$ and $Y_2^\mathrm{an}$ are locally biholomorphic over~$X^\mathrm{an}$ around the exceptional loci.
\end{lemma}

\begin{proof}
``$\Longrightarrow$''. The isomorphism $Y_1 \to Y_2$ induces a biholomorphism $Y_1^\mathrm{an} \to Y_2^\mathrm{an}$.

``$\Longleftarrow$''. The local biholomorphism extends to a unique biholomorphism $Y_1^\mathrm{an} \to Y_2^\mathrm{an}$ over~$X^\mathrm{an}$.
Now, suffices to show that if $f$ is a rational map of varieties such that its analytification is holomorphic, then $f$ is a morphism of varieties. For this, it suffices to show that if $f$ is a rational function on an affine variety $Z = \operatorname{Spec} A$ such that its analytification $f^\mathrm{an}$ is holomorphic, then $f \in A$. For this, we follow the argument in \cite{JS19}.

First, we show that $f$ is integral over~$A$. Let $\bar A$ be the integral closure of $A$ in its field of fractions.
Using the inclusions $\operatorname{Frac} A \to \operatorname{Frac} \bar A$ and $\mathcal O_{(\operatorname{Spec} A)^\mathrm{an}} \to \mathcal O_{(\operatorname{Spec} \bar A)^\mathrm{an}}$, we see that $f$ is a rational function on $\operatorname{Spec} \bar A$ and $f^\mathrm{an}$ is a holomorphic function on $(\operatorname{Spec} \bar A)^\mathrm{an}$. Therefore, $f^\mathrm{an}$ is bounded on every small analytic neighbourhood of any point of $(\operatorname{Spec} \bar A)^\mathrm{an}$. Therefore, the order of vanishing of $f$ along every prime divisor $D$ of $\operatorname{Spec} \bar A$ is non-negative. Since $\operatorname{Spec} \bar A$ is normal, we find $f \in \bar A$.

By \cite[Proposition~2.2]{JK20}, $A$ is integrally closed in $\mathcal O_{Z^\mathrm{an}}(Z^\mathrm{an})$. Since $f$ is holomorphic, we have $f \in \mathcal O_{Z^\mathrm{an}}(Z^\mathrm{an})$, and since $f$ is integral over~$A$, we find $f \in A$.
\end{proof}

\begin{corollary} \label{thm:a2a constructing weighted blowups}
Let $X$ be a variety, $P \in X$ a closed point and $U \subseteq X$ an affine open containing~$P$. Let $W \to Z$ be a weighted blowup of complex spaces with centre a point $Q \in Z$ such that $(X^{\mathrm{an}}, P)$ is locally biholomorphic to $(Z, Q)$. Then:
\begin{enumerate}[label=(\alph*), ref=\alph*]
\item \label{itm:a2a constructing weighted blowups - exists one} The construction in \cref{thm:a2a constructing wt-resp maps} gives a weighted blowup $Y \to X$ that is locally analytically equivalent to $W \to Z$.
\item \label{itm:a2a constructing weighted blowups - describes all} Every blowup $Y \to X$ that is locally analytically equivalent to $W \to Z$ is isomorphic over $X$ to a blowup $\hat Y \to X$ given in \cref{con:a2a construction} for some~$\psi$.
\end{enumerate}
\end{corollary}

\begin{proof}
(a) Suffices to consider the case where $Z$ is a complex subspace of~$\mathbb C^m$ and $Q$ is the origin. Using \cref{thm:a2a constructing wt-resp maps}, we find an isomorphism $U \to \hat U \subseteq \mathbb A^{n+m}$ and a choice of weights for the variables $x_1, \ldots, x_n, y_1, \ldots, y_m$ of $\mathbb A^{n+m}$ such that the weighted blowup of $\hat U \subseteq \mathbb A^{n+m}$ is locally analytically equivalent to $W \to Z$ by \cref{thm:wei wei-resp lifts}. By gluing, we find a weighted blowup $Y \to X$ which is locally analytically equivalent to $W \to Z$.

(b) Let $X^\mathrm{an} \to Z$ be local biholomorphism that lifts to the blown-up spaces. The construction in \cref{thm:a2a constructing wt-resp maps} gives a weighted blowup $\hat Y \to \hat X$, an isomorphism $\hat X \to X$ and a weight-respecting local biholomorphism $\hat X^\mathrm{an} \to Z$. Since both $X^\mathrm{an} \to Z$ and $\hat X^\mathrm{an} \to Z$ lift to the blown-up spaces, $\hat X^\mathrm{an} \to X^\mathrm{an}$ locally lifts to blown-up spaces. By \cref{thm:a2a isom over X algebraically iff analytically}, the isomorphism $\hat X \to X$ lifts to the blown-up spaces.
\end{proof}

\Cref{exa:a2a local analytic resolution of ODP} shows that \cref{thm:a2a constructing projective morphisms} and \cref{thm:a2a constructing weighted blowups}\labelcref{itm:a2a constructing weighted blowups - exists one} cannot always be true when we blow up a positive-dimensional closed complex subspace.

\begin{example} \label{exa:a2a local analytic resolution of ODP}
Let $X$ be a $\mathbb Q$-factorial 3-fold with a unique singular point $P$ which is an \emph{ordinary double point}, meaning a singularity isomorphic to $(\mathbb V(xy + zt), \bm 0) \subseteq (\mathbb C^4, \bm 0)$. Then locally analytically there exists a small resolution~$\varphi$, the blowup of the divisor $\mathbb V(x, z)$ of $\mathbb V(xy + zt)$ with exceptional locus a curve. On the other hand, since $X$ is $\mathbb Q$-factorial, there is no proper birational morphism $Y \to X$ from a smooth variety $Y$ which is locally analytically equivalent to~$\varphi$.
\end{example}

\section{Counting divisorial contractions}

We show that we can simplify \cref{thm:pre div contr cAn}.

\begin{theorem} \label{thm:cou simpler theorem}
\Cref{thm:pre div contr cAn} remains true if we
\begin{enumerate}
\item \label{itm:cou simpler cAn} replace \cref{itm:pre div contr cAn - gen f} with ``$\operatorname{wt} f = r_1 + r_2$'',
\item \label{itm:cou simpler cA1} replace \cref{itm:pre div contr cAn - cA1 f} with ``$(\mathbb V(f), \bm 0)$ is an $A_2$-singularity and $\operatorname{wt} f = 6$'' and
\item \label{itm:cou simpler cA2} replace \cref{itm:pre div contr cAn - cA2 f} with either ``$f = x^2 + y^2 + z^3 + xt^2$'' or with ``$(\mathbb V(f), \bm 0)$ is an $E_6$-singularity and $\operatorname{wt} f = 6$''.
\end{enumerate}
\end{theorem}

\begin{proof}
(1) Follows from \cite[Proposition~4.6]{Pae21}.

\medskip

(2) Let $f \in \mathbb C\{x, y, z, t\}$ be such that $\operatorname{wt} f = 6$ and $(\mathbb V(f), \bm 0)$ is an $A_2$-singularity.

If the coefficient of $yt$ is non-zero and the coefficient of $xy$ is zero in~$f$, then after a suitable coordinate change of the form $t \mapsto by + cz$, where $b$ and $c$ are complex numbers, the coefficients of $y^2$ and $yz$ will be zero in~$f$. This coordinate change is weight-respecting. Since $f$ has quadratic rank~$3$, after scaling, the quadratic part will be $yt + z^2$. Now $(\mathbb V(f), \bm 0)$ is an $A_2$-singularity if and only if the coefficient of $x^3$ is non-zero, which cannot happen since $\operatorname{wt} f = 6$.

Therefore, since the quadratic rank of $f$ is~$3$, the coefficient of $xy$ is non-zero. After a suitable coordinate change of the form $x \mapsto x + by + cz + dt$, where $a, b, c, d$ are complex numbers, the coefficients of $y^2$, $yz$ and $yt$ will be zero. This coordinate change is weight-respecting. Now the coefficient of $z^2$ must be non-zero. After scaling, the quadratic part of $f$ will be $xy + z^2$. We see that $(\mathbb V(f), \bm 0)$ is an $A_2$-singularity if and only if the coefficient of $t^3$ is non-zero.

After scaling, the least weight non-zero quasihomogeneous part of $f$ with respect to weights $\bm w = (3, 3, 3, 2)$ will be $xy + z^2 + t^3$. \Cref{thm:wei wei-resp reduce simple sing to normal form} gives a weight-respecting automorphism $\Psi$ such that $\Psi(f) = xy + z^2 + t^3$.

\medskip

(3) To begin, we show that $f \in \mathbb C\{x, y, z, t\}$ satisfying \cref{itm:pre div contr cAn - cA2 f} of \cref{thm:pre div contr cAn} defines an $E_6$-singularity at the origin. Let $\Psi$ be the automorphism of $\mathbb C\{x, y, z, t\}$ given by composing $x \mapsto x/\sqrt{1 - c^2}$ with $y \mapsto y - c(x + p_{\operatorname{wt} = 3})$. Defining $p', g' \in \mathbb C\{z, t\}$ by $p = 1/2 (\sqrt{1 - c^2}) p' + c^2 p_{\operatorname{wt} = 3}$ and $g = g' + c^2 + p_{\operatorname{wt} = 3}^2 - z^3$, we find that $\Psi(f)$ is equal to $x^2 + y^2 + xp' + g'$, where $\operatorname{wt} p'$ is~$2$, $\operatorname{wt} g'$ is~$6$ and all monomials of weight greater than $3$ have coefficient zero in~$p'$.
Let $\Phi$ be the coordinate change $x \mapsto x - p'/2$ composed with a suitable scaling of the variable~$t$. Then the least weight non-zero quasihomogeneous part of $\Phi(\Psi(f))$ will be $x^2 + y^2 + z^3 + t^4$ under the weights $\bm w = (6, 6, 4, 3)$. Using the lemma in \cite[§12.6]{AGZV85} or \cref{thm:wei wei-resp reduce simple sing to normal form} we find that $\Phi(\Psi(f))$ is right equivalent to $x^2 + y^2 + z^3 + t^4$, proving that $f$ defines an $E_6$-singularity at the origin.

Now let $f \in \mathbb C\{x, y, z, t\}$ be such that $\operatorname{wt} f = 6$ and $(\mathbb V(f), \bm 0)$ is an $E_6$-singularity.

First, we show that the coefficient of $xz$ in $f$ is zero. We remind that the quadratic rank of a convergent power series defining a $3$-dimensional $E_6$-singularity is~$2$. If the coefficient of $xz$ is non-zero, since $f$ has quadratic rank $2$, the coefficient of $y^2$ must be zero. After a suitable coordinate change of the form $z \mapsto ax + by + cz$, where $a, b$ and $c$ are complex numbers and $c$ is non-zero, the quadratic part of $f$ will be $xz$. By \cref{thm:wei x1*x2 weight-respecting splitting} after a weight-respecting coordinate change $f$ will be of the form $xz + h$ where $h \in \mathbb C\{y, t\}$. If the coefficient of $y^2 t$ is non-zero, then by \cite[Theorem~I.2.51]{GLS07} $(\mathbb V(f), \bm 0)$ is either a $D_k$-singularity or a non-isolated singularity, a contradiction. So the coefficient of $y^2 t$ is zero. Since $(\mathbb V(f), \bm 0)$ is a $cA_2$-singularity, after scaling the 3-jet of $f$ will be $xz + y^3$. If the coefficient $d$ of $y t^3$ is non-zero, then \cref{thm:wei wei-resp reduce simple sing to normal form} with $\bm w := (9, 6, 9, 4)$ and $f_0 := xz + y^3 + d yt^3$ implies that $(\bm V(f), \bm 0)$ is an $E_7$-singularity, a contradiction. Therefore, the coefficient of $y t^3$ is zero. Now \cite[Theorem~I.2.55(2)]{GLS07} shows that $(\mathbb V(f), \bm 0)$ is not a simple singularity, a contradiction.

Second, we show that the coefficient of $y^2$ is non-zero. If the coefficient of $y^2$ is zero, then the coefficient of $xy$ must be non-zero. After a suitable coordinate change of the form $y \mapsto ax + y + h$, where $a \in \mathbb C$ and $h \in \mathbb C\{x, y, z, t\}$ has multiplicity at least $2$ or is zero, the only monomial that is divisible by $x$ and has non-zero coefficient in $f$ will be $xy$. Now the weight of $f$ is at least $5$ and the monomials of weight $5$ that have a non-zero coefficient in $f$ are in the set $\{zt^3, t^5\}$. After a suitable coordinate change of the form $x \mapsto x + h'$, where $h' \in \mathbb C\{y, z, t\}$ has multiplicity at least $2$ or is zero, the only monomial in the ideal $(x, y)$ that has non-zero coefficient in $f$ will be $xy$. The weight of $f$ is still at least $5$ and the monomials of weight $5$ are still in the set $\{zt^3, t^5\}$. Since $(\mathbb V(f), \bm 0)$ is a $cA_2$-singularity, the coefficient of $z^3$ is non-zero. If the coefficient of $zt^3$ is non-zero, then \cref{thm:wei wei-resp reduce simple sing to normal form} with $\bm w := (9, 9, 6, 4)$ and $f_0 := axy + bz^3 + czt^3$, where $a, b, c \in \mathbb C$ are non-zero, shows that $(\mathbb V(f), \bm 0)$ is an $E_7$-singularity, a contradiction. So the coefficient of $zt^3$ is zero. If the coefficient of $t^5$ is non-zero, then \cref{thm:wei wei-resp reduce simple sing to normal form} with $\bm w := (15, 15, 5, 3)$ and $f_0 := axy + bz^3 + ct^5$, where $a, b, c \in \mathbb C$ are non-zero, shows that $(\bm V(f), \bm 0)$ is an $E_8$-singularity, a contradiction. Therefore, $f - axy$ belongs to the ideal $(z, t^2)$ of $\mathbb C\{z, t\}$. By \cite[Theorem~I.2.55(2)]{GLS07} $(\mathbb V(f), \bm 0)$ is not a simple singularity, a contradiction.

Next, we show that the coefficient of $xt^2$ is non-zero. If the coefficient of $xt^2$ is zero, then after a suitable linear weight-respecting coordinate change the quadratic part of $f$ will be $xy + y^2$. Now after a suitable weight-respecting coordinate change of the form $y \mapsto y + h$, where $h \in \mathbb C\{x, y, z, t\}$ has multiplicity at least 2 or is non-zero, followed by an application of \cref{thm:wei wei-resp power series to poly}, the only monomial with non-zero coefficient in $f$ that is divisible by $x$ will be $xy$. After a suitable coordinate change of the form $x \mapsto x + h'$, where $h' \in \mathbb C\{x, y, z, t\}$, the only monomial in the ideal that has non-zero coefficient in $f$ will be $xy$ and the weight of $f$ will still be~$6$. By \cite[Theorem~I.2.55(2)]{GLS07} $(\mathbb V(f), \bm 0)$ is not a simple singularity, a contradiction.

Now, after a suitable linear weight-respecting coordinate change, the quadratic part of $f$ will be $x^2 + y^2$. Using a suitable weight-respecting coordinate change of the form $x \mapsto x + h$ and $y \mapsto y + h'$, where $h, h' \in \mathbb C\{x, y, z, t\}$, followed by an application of \cref{thm:wei wei-resp power series to poly} the power series $f$ will have the form
\begin{equation} \label{equ:cou x2 y2 xp g}
f = x^2 + y^2 + xp + g,
\end{equation}
where $p \in \mathbb C\{z, t\}$ has only monomials of weight $2$ and $3$, the coefficient of $t^2$ in $p$ is~$1$ and the coefficient of $z^3$ in $g \in \mathbb C\{z, t\}$ is~$1$.

Finally, we show that there exists a weight-respecting automorphism $\Psi$ of $\mathbb C\{x, y, z, t\}$ such that $\Psi(f) = x^2 + y^2 + z^3 + xt^2$, where $f$ is given by \cref{equ:cou x2 y2 xp g}. The least weight non-zero quasihomogeneous part of $g - p^2/4$ under the weights $(4, 3)$ is $z^3 - t^4/4$. By \cref{thm:wei wei-resp reduce simple sing to normal form} there exists an automorphism $\Phi$ of $\mathbb C\{z, t\}$ such that $\Phi(g - p^2/4)$ is equal to $z^3 - t^4/4$, $\Phi(z) - z$ is in the ideal $(z, t)^2$ and $\Phi(t) - t$ is in the ideal $(z, t^2)$. So $\Phi$ is weight-respecting with respect to weights $(2, 1)$. We find that
\begin{equation} \label{equ:cou showing f is x2 y2 xt2 z3}
(\Phi(p) - t^2) (\Phi(p) + t^2) = 4(\Phi(g) - z^3)
\end{equation}
is either zero or has weight at least~$6$. The term $\Phi(p) + t^2$ has weight $2$ since the coefficient of $t^2$ is~$2$. Therefore, $\Phi(p) - t^2$ is either zero or has weight at least~$4$. Applying $\Phi^{-1}$ to \cref{equ:cou showing f is x2 y2 xt2 z3}, we find that $\Phi^{-1}(t^2) - p$ is also either zero or has weight at least~$4$. Now suffices to choose $\Psi$ to be
\[
\Psi(x) := x + \frac{t^2 - \Phi(p)}{2},~
\Psi(y) := y,~
\Psi(z) := \Phi(z),~
\Psi(t) := \Phi(t).\qedhere
\]
\end{proof}

\begin{lemma} \label{thm:cou div contr same type are isom}
Let $n$ be a positive integer. Let $P$ be a $cA_n$-point of a $\mathbb Q$-Gorenstein variety $X$ with terminal singularities. Then any two divisorial contractions to $X$ with centre $P$ are locally analytically equivalent if they are either
\begin{enumerate}
\item \label{itm:cou div contr same type - cAn} both of type (1) with the same weights $(r_1, r_2, a, 1)$,
\item both of type (2) or
\item both of type (3).
\end{enumerate}
\end{lemma}

\begin{proof}
Case (1) is \cite[Proposition~4.7]{Pae21}, case (2) is clear and case (3) follows from \cref{thm:cou simpler theorem}\labelcref{itm:cou simpler cA2}.
\end{proof}

We describe conditions for the existence of divisorial contractions to $X$ with centre $P$ of types (1), (2) and (3) of \cref{thm:pre div contr cAn}.

\begin{lemma} \label{thm:cou exist of div contr}
Let $P$ be a $cA_n$-point of a $\mathbb Q$-Gorenstein variety $X$ with terminal singularities.
\begin{enumerate}[label=(\alph*), ref=\alph*]
\item \label{itm:cou exist of div contr - lower r1 r2 a 1} If there exists a divisorial contraction of type (1) to $X$ with center $P$ which is an $(r_1, r_2, a, 1)$-blowup, then for all $a' \in \{1, \ldots, a\}$ and for all $r_1' \in \{1, \ldots, a'(n+1) - 1\}$ such that $a'$ is coprime to both $r_1'$ and $r_2' := a'(n+1) - r_1'$ there exists a divisorial contraction of type (1) which is an $(r_1', r_2', a', 1)$-blowup.
\item \label{itm:cou exist of div contr - upper bound for a} There is a positive integer $N$ such that there is no divisorial contraction of type (1) to $X$ with center $P$ which is an $(r_1, r_2, a, 1)$-blowup where $a > N$.
\item \label{itm:cou exist of div contr - cA1 a geq 2} If $n = 1$, then there exists a divisorial contraction of type (1) which is an $(r_1, r_2, a, 1)$-blowup if and only if $(X^{\mathrm{an}}, P)$ is an $A_k$-singularity where $k \geq a$.
\item \label{itm:cou exist of div contr - cA1 special} If $n = 1$, then there exists a divisorial contraction of type (2) if and only if $(X^{\mathrm{an}}, P)$ is the $A_2$-singularity.
\item \label{itm:cou exist of div contr - cA2 a geq 2} If $n = 2$, then there exists a divisorial contraction of type (1) with $a = 2$ if and only if $(X^{\mathrm{an}}, P)$ is not a simple singularity.
\item \label{itm:cou exist of div contr - cA2 special} If $n = 2$, then there exists a divisorial contraction of type (3) if and only if $(X^{\mathrm{an}}, P)$ is an $E_6$-singularity.
\end{enumerate}
\end{lemma}

\begin{proof}
(a) If $f$ is of the form $xy + g$ and the weight of $g \in \mathbb C\{z, t\}$ is $r_1 + r_2$ with respect to the weights $(r_1, r_2, a, 1)$, then the weight of $g$ is also $r_1' + r_2'$ with respect to the weights $(r_1', r_2', a', 1)$.

(b) By \cite[Corollary~I.2.18]{GLS07} or \cite[§12.2]{AGZV85} the Milnor number of $xy + g_{\operatorname{wt} = r_1 + r_2}$ is at least $n((n+1)a - 1)$. On the other hand, the isolated singularity $(X^{\mathrm{an}}, P)$ has finite Milnor number.

(e) By \cite[Theorem~I.2.55(2)]{GLS07} a $cA_2$-singularity $(\mathbb V(f), \bm 0)$, where $f$ is in $\mathbb C\{x, y, z, t\}$, is not contact simple if and only if there is an automorphism $\Psi$ of $\mathbb C\{x, y, z, t\}$ such that $\Psi(f) = xy + g(z, t)$, where $g$ is in the ideal $z, t^2$ of $\mathbb C\{z, t\}$.

Parts (c), (d) and (f) follow from the definition of simple singularities (\cref{def:pre simple singularity}).
\end{proof}

It is known that there are only finitely many divisorial contractions with discrepancy at most~$1$, see \cite[below Theorem~1.2]{Kaw05}. I have added a proof here since I have not found a proof in the literature. The precise statement is as follows:

\begin{proposition} \label{thm:cou finitely many if disc at most 1}
Let $X$ be a $\mathbb Q$-Gorenstein variety with terminal singularities. Then there are only finitely many divisorial contractions to~$X$ with discrepancy at most~$1$.
\end{proposition}

\begin{proof}
Let $f\colon Y \to X$ be a resolution of singularities with exceptional locus of pure codimension~$1$. Let $v$ be the valuation on the function field $\mathbb C(X)$ given by the exceptional divisor of a divisorial contraction to~$X$. Then $v$ is equal to the valuation given by a prime divisor $D$ on a normal variety $Z$ with a proper birational morphism $Z \to Y$. The centre of $D$ on $Y$ is necessarily contained in an exceptional prime divisor of~$f$. We see that if the discrepancy of $D$ is at most~$1$, then the centre of $D$ on $Y$ necessarily coincides with an exceptional prime divisor of~$f$. So $v$ is equal to the valuation given by one of the finitely many exceptional prime divisor of~$f$. The \lcnamecref{thm:cou finitely many if disc at most 1} follows from the fact that any two divisorial contractions whose exceptional divisors define the same valuation are isomorphic over~$X$, see \cite[Lemma~3.4]{Kaw01}.
\end{proof}

\begin{theorem} \label{thm:cou counting divisorial contractions}
Let $n$ be a positive integer. Let $P$ be a point of a $\mathbb Q$-Gorenstein variety $X$ with terminal singularities. We count the number of divisorial contractions to $X$ with centre $P$.
\begin{enumerate}[label=(\alph*), ref=\alph*]
\item If $(X^{\mathrm{an}}, P)$ is smooth, then there are uncountably many divisorial contractions up to isomorphism over $X$ and countably many up to local analytic equivalence.
\item \label{itm:cou counting divisorial contractions - disc 1 finitely many} If $(X^{\mathrm{an}}, P)$ is a $cA_n$-singularity that admits only discrepancy 1 divisorial contractions, then there are exactly $n$ divisorial contractions up to isomorphism over $X$ and exactly $\lceil n/2 \rceil$ up to local analytic equivalence, where $\lceil r \rceil$ denotes the smallest integer greater than or equal to the real number~$r$.
\item If $(X^{\mathrm{an}}, P)$ is a $cA_n$-singularity that admits a divisorial contraction with discrepancy $\geq 2$, then there are uncountably many divisorial contractions up to isomorphism over $X$ and finitely many up to local analytic equivalence.
\end{enumerate}
\end{theorem}

\begin{proof}
(a) By \cref{thm:pre div contr smooth} there are countably many divisorial contractions up to local analytic equivalence. Since the automorphism $\Psi$ of $\mathbb C\{x, y, z\}$ given by $z \mapsto z + ax$, where $a \in \mathbb C$ is non-zero, does not lift to an isomorphism of the blown-up spaces when performing a $(1, 1, 2)$-blowup, there are uncountably many divisorial contractions up to isomorphism over~$X$.

(b) Similarly to the proof of \cite[Theorem~6.4]{Hay99} we can show that there are exactly $n$ local analytic germs of divisorial contractions up to isomorphism over~$X^{\mathrm{an}}$. Note that the last sentence in the statement of \cite[Theorem~6.4]{Hay99} contains a typo, it should say: ``Furthermore, there are exactly $k$ divisors with discrepancies $1/m$ over~$X$'' (the symbol $k$ was missing). The global algebraic divisorial contractions are constructed using \cref{thm:a2a constructing projective morphisms} or \cref{thm:a2a constructing weighted blowups}. To see that there are exactly $\lceil n/2 \rceil$ divisorial contractions up to local analytic equivalence, note that $(x, y, z, t) \mapsto (y, x, z, t)$ is weight-respecting with respect to the weights $(r_1, r_2, a, 1)$ and $(r_2, r_1, a, 1)$.

(c) It follows from \cref{thm:cou exist of div contr}\labelcref{itm:cou exist of div contr - upper bound for a} and \cref{thm:cou div contr same type are isom} that there are only finitely many divisorial contractions up to local analytic equivalence.

If $(X^{\mathrm{an}}, P)$ is not an $E_6$-singularity, then there exists a divisorial contraction of type (1) of \cref{thm:cou simpler theorem} with $a > 1$. By \cref{thm:cou exist of div contr}\labelcref{itm:cou exist of div contr - lower r1 r2 a 1} there exists a divisorial contraction with $r_1 = 1$. Let $f \in \mathbb C\{x, y, z, t\}$ be as in \cref{itm:pre div contr cAn - gen f}. For any $c \in \mathbb C$ there exists an automorphism $\Phi_c$ of $\mathbb C\{x, y, z, t\}$ that fixes $f$ given by
\[
\Phi_c(x) := x,~
\Phi_c(y) := y + h,~
\Phi_c(z) := z + cx,~
\Phi_c(t) := t,
\]
where $h \in \mathbb C\{x, y, z, t\}$ depends on~$f$. Each automorphism $\Phi_c$ defines a divisorial contraction of the analytic germ $(X^{\mathrm{an}}, P)$, naming composing the divisorial contraction to $X$ with the precomposition with~$\Phi_c$. The composition $\Phi_{c'} \circ \Phi_c^{-1}$ is weight-respecting with respect to weights $(1, r_2, a, 1)$ if and only if $c = c'$. We can check on the affine patch $t \neq 0$ of the $(1, r_2, a, 1)$-blown-up space that the biholomorphic map germ corresponding to $\Phi_{c'} \circ \Phi_c^{-1}$ lifts to an isomorphism of the blown-up spaces if and only if $c = c'$. Thus there are uncountably many analytic germs of $(1, r_2, a, 1)$-blowups to $X$ with centre~$P$. By \cref{thm:a2a constructing projective morphisms} or \cref{thm:a2a constructing weighted blowups} each such analytic germ extends to a divisorial contraction to $X$ with centre~$P$.

If $(X^{\mathrm{an}}, P)$ is an $E_6$-singularity, then for any complex number $v \in \mathbb C$, any square root $w$ of $1 - v^2$ and any $u \in \{-1, 1\}$, the automorphism $\Psi_{u, v, w}$ of $\mathbb C\{x, y, z, t\}$ given by
\[
\begin{aligned}
\Psi_{u, v, w}(x) & := vx + wy + (v - 1)t^2/2,\\
\Psi_{u, v, w}(y) & := uwx - uvy + uwt^2/2,
\end{aligned}
\quad
\begin{aligned}
\Psi_{u, v, w}(z) & := z\\
\Psi_{u, v, w}(t) & := t
\end{aligned}
\]
fixes $x^2 + y^2 + z^3 + xt^2$.
Note that $\Psi_{u', v', w'} \circ \Psi_{u, v, w}^{-1}$ is weight-respecting with respect to weights $(4, 3, 2, 1)$ if and only if $v' = v$ and $w' = u u' w$. We can check that the biholomorphic map germ corresponding to $\Psi_{u', v', w'} \circ \Psi_{u, v, w}^{-1}$ lifts to an isomorphism of the blown-up spaces if and only if $v' = v$ and $w' = u u' w$. Similarly to the previous case, this shows that there are uncountably many divisorial contractions of type \labelcref{itm:pre div contr cAn - cA2} to $X$ with center $P$.
\end{proof}

\subsubsection*{Acknowledgements.} I would like to thank Masayuki Kawakita, Takuzo Okada and Yuki Yamamoto for useful discussions. I would like to thank Ziquan Zhuang for noticing a mistake in an earlier draft. Supported by the Simons Investigator Award HMS, National Science Fund of Bulgaria, National Scientific Program ``Excellent Research and People for the Development of European Science'' (VIHREN), Project No.~KP-06-DV-7.

\providecommand{\bysame}{\leavevmode\hbox to3em{\hrulefill}\thinspace}
\providecommand{\MR}{\relax\ifhmode\unskip\space\fi MR }
\providecommand{\MRhref}[2]{%
  \href{http://www.ams.org/mathscinet-getitem?mr=#1}{#2}
}
\providecommand{\href}[2]{#2}

\end{document}